\documentclass[reqno]{amsart}
\pdfoutput=1
\usepackage{amsmath}
\usepackage{amsthm}
\usepackage{amsxtra}
\usepackage{amsfonts,amssymb}
\usepackage{graphicx}
\usepackage{marginnote}
\usepackage[top=1.5in, bottom=1.5in, left=1in, right=1in]{geometry}
\newtheorem{Th}{Theorem}
\newtheorem{Prop}[Th]{Proposition}
\newtheorem{Lm}[Th]{Lemma}
\newtheorem{Co}[Th]{Corollary}

\newtheorem*{mainthm}{Main Theorem}{\bf}{\it}

\def\be{\begin{eqnarray}} \def\ee{\end{eqnarray}} \def\bes{\begin{eqnarray*}}
\def\ees{\end{eqnarray*}} \def\bit{\begin{itemize}} \def\eit{\end{itemize}}
\def\ba{\begin{array}{ll}} \def\ea{\end{array}}

\def\mn{\mathbb{N}}

\def\mq{\mathbb{Q}}

\def\mc{\mathbb{C}}  

\def\mr{\mathbb{R}}
\def\mrc{\mathbb{R}_\mathcal{C}}

\def\md{\mathbb{D}}

\def\mh{\mathbb{H}}

\def\cp{\mathcal{P}}

\def\cl{\mathcal{L}}

\def\cR{\mathcal{R}}

\def\ca{\mathcal{A}}
\def\cd{\mathcal{D}}

\newtheorem{Def}[Th]{Definition}

\newtheorem{Rem}[Th]{Remark} \def\C{\mathbb{C}}  \def\mn{\mathbb{N}}
\def\le{\leqslant} \def\ge{\geqslant}  \def\md{\mathbb{D}}

\def\d{\text{dist}} \def\re{\text{Re}} \def\im{\text{Im}}
\def\diam{\text{diam}}
\def\i{\mathrm{i}}
\def\inter{\text{Int}}

\def\id{\text{Id}}

\newcommand{\dd}{\mathrm{d}}

\title{Poly-time Computability of the Feigenbaum Julia set.}
\author{Artem Dudko and Michael Yampolsky}
\date{}
\begin{document}
\begin{large}
\maketitle
\begin{abstract} We present the first example of a poly-time computable Julia set with a recurrent critical point:
we prove that the Julia set of the Feigenbaum map is computable in polynomial time.

\end{abstract}

%%%%\tableofcontents

\section{Introduction.}
Informally speaking, a compact set $K$ in the plane is computable if there exists an algorithm to draw it on
a computer screen with any desired precision. Any computer-generated picture is a finite collection of pixels.
If we fix a specific pixel size (commonly taken to be $2^{-n}$ for some $n$) then to accurately draw the set within
one pixel size, we should fill in the pixels which are close to the set (for instance, within distance $2^{-n}$ from it),
and leave blank the pixels which are far from it (for instance, at least $2^{-(n-1)}$-far). Thus, for the set $K$ to be 
computable, there has to exist an algorithm which for every square of size $2^{-n}$ with dyadic rational vertices 
correctly decides whether it should be filled in or not according to the above criteria. 
We say that a computable set has a polynomial time complexity (is {\it poly-time})
if there is an algorithm which does this in a time bounded by a polynomial function of the precision parameter $n$, independent
of the choice of a pixel.

When we talk of computability of the Julia sets of a rational map,
the algorithm drawing it is supposed to have access to the values of the coefficients of the map (again with an arbitrarily
high precision). Computability of Julia sets has been explored in depth by M.~Braverman and the second author
(see monograph \cite{BY06} and references therein). They have shown, in particular, that there exist quadratic polynomials
$f_c(z)=z^2+c$ with explicitly computable parameters $c$ whose Julia sets $J_c$ are not computable. Such parameters
are rare, however; for almost every $c\in\mc$ the set $J_c$ is computable. In \cite{BBY06} it was shown that there exist
computable quadratic Julia sets with an arbitrarily high time complexity. On the other hand, hyperbolic Julia sets are
poly-time \cite{Bra04,Ret05}. The requirement of hyperbolicity may be weakened significantly. The first author has shown \cite{Dudko}
that maps with non-recurrent critical orbits have poly-time Julia sets. However, even in the quadratic family $f_c$ it is not
at present known if $J_c$ is poly-time for a typical value of $c$ (see the discussion in \cite{Dudko}). 

Until now, no examples of poly-time computable Julia sets with a recurrent critical point have been known. 
In this note we present the first such example. It is given by
 perhaps the most famous quadratic map of all -- the Feigenbaum polynomial $f_{c_*}$.
The Feigenbaum map is infinitely renormalizable under period-doubling, and its renormalizations converge to a fixed point of the renormalization
operator. Historically, this is the first instance of renormalization in Complex Dynamics (see \cite{Lyu} for an overview of the history of the
subject). As follows from \cite{BBY07}, the Julia set of $f_{c_*}$ is computable. However, infinite renormalizability implies, in particular,
that we cannot expect to find any hyperbolicity in the dynamics of $f_{c_*}$ to make the computation fast.
We find a different hyperbolic dynamics, however, to speed up the computation -- the dynamics of the renormalization operator itself.
In a nutshell, this is the essense of the poly-time algorithm described in this paper.

The details, however, are quite technical and analytically involved. To simplify the exposition, we prove poly-time computability of the Julia set
not of the map $f_{c_*}$ itself, but of the Feigenbaum renormalization fixed point $F$. The map $F$ is not a quadratic polynomial,
but it is quadratic-like, it is conjugate to $f_{c_*}$, and its Julia set is homeomorphic to that of $f_{c_*}$. 
There are two advantages to working with $F$, as opposed to $f_{c_*}$. Firstly, the renormalization-induced self-similarity of the Julia set
of $F$ is exactly, rather than approximately linear. This allows us to streamline the arguments somewhat, making them easier
to follow. More importantly, as we show, the map $F$ itself is poly-time computable (an efficient algorithm for computing $F$ is due to
O.~Lanford \cite{Lan}). Hence, our main result -- poly-time computability of the Julia set of $F$ -- can be stated without the use of
an oracle for the map $F$.

We now proceed to give the detailed definitions and 
precise statements of our main results.

\subsection{Preliminaries on computability}
In this section we give a very brief review of computability and complexity of sets.
For details we refer the reader to the monograph \cite{BY08}.
 The notion of computability relies on the concept of a Turing Machine (TM) \cite{Tur},
 which is a commonly accepted way of formalizing
the definition of an algorithm.
A precise description of a Turing Machine is quite technical and we
do not give it here, instead referring the reader to any text on Computability Theory (e.g. \cite{Pap} and \cite{Sip}). 
The computational power of a Turing Machine is provably equivalent to that of a computer program running on a RAM
computer with an unlimited memory.

\begin{Def}\label{comp fun def} 
A function $f:\mn\rightarrow \mn$ is called computable, if there exists
a TM which takes $x$ as an input and outputs $f(x)$.
\end{Def}
Note that Definition \ref{comp fun def} can be naturally extended to
functions on arbitrary countable sets, using a convenient
identification with $\mathbb{N}$. The following definition of a computable real number is
due to Turing \cite{Tur}:
\begin{Def} A real number $\alpha$ is called computable if there is a
 computable function $\phi:\mathbb{N}\rightarrow \mathbb{Q}$,
 such that for all $n$ $$\left|\alpha-\phi(n)\right|<2^{-n}.$$
\end{Def} 
 The set of computable reals is denoted by $\mathbb{R}_\mathcal{C}$. Trivially, $\mq\subset\mrc$. Irrational numbers
such as $e$ and $\pi$ which can be computed with an arbitrary precision also belong to $\mrc$. However, since there
exist only countably many algorithms, the set $\mrc$ is countable, and hence a typical real number is not
computable. 

 The set of computable complex numbers is
defined by
$\mathbb{C}_\mathcal{C}=\mathbb{R}_\mathcal{C}+i\mathbb{R}_\mathcal{C}$.
Note that  $\mathbb{R}_\mathcal{C}$ (as well as $\mathbb{C}_\mathcal{C}$)
considered with the usual arithmetic operation forms a  field. 

To define computability of functions of real or complex variable we need to introduce the concept of an oracle:
\begin{Def}
A function $\phi:\mn\to\mq+i\mq$ is an oracle for $c\in\mc$ if for every $n\in\mn$ we have
$$|c-\phi(n)|<2^{-n} .$$
\end{Def}
A TM equipped with an oracle (or simply an {\it oracle TM}) may query the oracle by reading the value of $\phi(n)$ for an arbitrary $n$.
\begin{Def}
Let $S\subset \mc$. A function $f:S\to \mc$ is called computable if there exists an oracle TM $M^\phi$ with a single natural
input $n$ such that if $\phi$ is an oracle for $z\in S$ then $M^\phi$ outputs $w\in \mq+i\mq$ such that
$$|w-f(z)|<2^{-n} .$$
\end{Def}
We say that a function $f$ is {\it poly-time computable} if in the above definition the algorithm $M^\phi$ can be made to run
in time bounded by a polynomial in $n$, independently of the choice of a point $z\in S$ or an oracle representing this 
point. Note that when calculating the running time of $M^\phi$, querying $\phi$ with precision $2^{-m}$ counts as 
$m$ time units. In other words, it takes $m$ ticks of the clock to read the argument of $f$ with precision $m$ (dyadic)
digits.

 Let $d(\cdot,\cdot)$ stand for Euclidean distance between points or sets in $\mathbb{R}^2$.
 Recall the definition of the {\it Hausdorff distance} between two sets:
$$d_H(S,T)=\inf\{r>0:S\subset U_r(T),\;T\subset U_r(S)\},$$
where $U_r(T)$ stands for the $r$-neighborhood of $T$:
$$U_r(T)=\{z\in \mathbb{R}^2:d(z,T)\leqslant r\}.$$ We call
a set $T$ a {\it $2^{-n}$ approximation} of a bounded set $S$, if
$d_H(S,T)\leqslant 2^{-n}$. When we try to draw a $2^{-n}$ approximation $T$ of a set $S$
 using a computer program, it is convenient to
 let $T$ be a finite
 collection of disks of radius $2^{-n-2}$ centered at points of the form $(i/2^{n+2},j/2^{n+2})$
 for $i,j\in \mathbb{Z}$.  We will call such a set {\it dyadic}.
A dyadic set $T$ can be described using a function
\begin{eqnarray}\label{comp fun}h_S(n,z)=\left\{\begin{array}{ll}1,&\text{if}\;\;d(z,S)\leqslant 2^{-n-2},
\\0,&\text{if}\;\;d(z,S)\geqslant 2\cdot 2^{-n-2},\\
0\;\text{or}\;1&\text{otherwise},
\end{array}\right.\end{eqnarray} where $n\in \mathbb{N}$ and $z=(i/2^{n+2},j/2^{n+2}),\;i,j\in \mathbb{Z}.$\\
%%\begin{figure}\centering\epsfig{file=CompPic1.eps,width=.70\linewidth}\caption{Values
%%of the function $h_S$.} \end{figure}
 \\ Using this
function, we define computability and computational
complexity of a set in $\mathbb{R}^2$ in the following way.
\begin{Def}\label{DefComputeSet} A bounded set $S\subset
\mathbb{R}^2$ is called computable in time $t(n)$ if there
is a TM, which computes values of a function $h(n,\bullet)$
of the form (\ref{comp fun}) in time $t(n)$. We say that
$S$ is poly-time computable, if there exists a polynomial
$p(n)$, such that $S$ is computable in time $p(n)$.
\end{Def}

\subsection{Renormalization and the Feigenbaum map $F$.}
In this section we recall some important notions from Renormalization Theory for quadratic like maps  and introduce the Feigenbaum map $F$. We refer the reader to \cite{DH} and \cite{McM} for details on renormalization and to \cite{Bu} for properties of the Feigenbaum map.
\begin{Def}\label{DefQL} A quadratic-like map is a ramified covering $f:U\to V$ of degree $2$, where $U\Subset V$ are topological disks. For a quadratic-like map $f$
we define its filled Julia set $K(f)$ and Julia set $J(f)$ as follows
\begin{equation}\label{EqJuliaDef} K(f)=\{z\in U:f^n(z)\in U\text{ for every }n\in\mn\},\;\;J(f)=\partial K(f).
\end{equation}
\end{Def}
\noindent
Without lost of generality, we will assume that the critical point of a quadratic-like map $f$ is at the origin.

Let $P_c(z)=z^2+c$. Then for $R$ large enough the restriction of $P_c$ onto the disk $D_R(0)=\{z:|z|<R\}$ is a quadratic-like map.
\begin{Def} Two quadratic-like maps $f_1$ and $f_2$ are said to be hybrid equivalent if there is a quasiconfromal map $\psi$ between neighborhoods of $K(f_1)$
and $K(f_2)$ such that $\bar\partial\psi=0$ almost everywhere on $K(f_1)$.
\end{Def}
\noindent
Douady and Hubbard proved the following:
\begin{Th}(Straightening Theorem) Every quadratic-like map $f$ is hybrid equivalent to a quadratic map $P_c$. If the Julia set $J(f)$ is connected, then the map $P_c$ is unique.
\end{Th}
\noindent
Let $f$ be a quadratic-like map with connected Julia set.  The parameter $c$ such that $P_c$ is hybrid equivalent to $f$ is called the \emph{inner class} of $f$ and is denoted by $I(f)$.

Recall the notion of renormalization.
\begin{Def}\label{DefRenormMap} Let $f:U\to V$ be a quadratic-like map. Assume that there exists a number $n>1$ and a topological disk
 $U'\ni 0$ such that $f^n_{|U'}$ is a quadratic-like map. Then $f$ is called renormalizable with period $n$. The map $\cR f:=f^n_{|U'}$ is called a renormalization of $f$.
\end{Def}
\noindent
Observe that the domain $U'$ of $\cR f$ from the definition above is not uniquely defined.  Therefore, it is more natural to consider renormalization of germs rather than maps.
\begin{Def}\label{DefEquivGerms} We will say that two quadratic like maps $f$ and $g$ with connected Julia sets define the same germ $[f]$ of quadratic-like map if $J(f)=J(g)$ and $f\equiv g$ on a neighborhood of the Julia set.
\end{Def}
\noindent
We define the renormalization operator $\cR_2$ of period $2$ as follows.
\begin{Def}\label{DefRenormGerms} Let $[f]$ be a germ of a quadratic-like map renormalizable with period $2$. Let $U'\supset 0$ be such that $g:=f^2_{|U'}$ is a quadratic-like map. We set \[\cR_2[f]=[\alpha^{-1}\circ g\circ \alpha],\] where $\alpha(z)=g(0)z$.
\end{Def}
\noindent We have introduced the normalization $\alpha(z)$ in order to have that the critical value of the renormalized germ is at $1$.

We recall, that
the Feigenbaum parameter value $c_{Feig}\in\mr$ is defined as the limit of the parameters $c_n\in\mr$ for which 
the critical point $0$ of the quadratic polynomial is periodic with period $2^n$. 
The Feigenbaum polynomial is the map
$$P_{Feig}(z)=z^2+c_{Feig}.$$
The next theorem follows from the celebrated work of Sullivan (see \cite{dMvS}):
\begin{Th}\label{ThConvRenorm} The sequence of germs $\cR_2([P_{Feig}])$ converges to a point $[F]$. The germ $[F]$ is a unique fixed point of the renormalization operator $\cR_2$ and is hybrid equivalent to $[P_{Feig}]$.
\end{Th}
\noindent
\subsection{The main result.}
Note that the germ $[F]$ from Theorem \ref{ThConvRenorm} has a well-defined quadratic-like Julia set $J_F$. We state:
\begin{mainthm}
The Julia set $J_F$ is poly-time computable.
\end{mainthm}

\section{The structure of the Feigenbaum map $F$.}
In this section we show how to compute the coefficients of the map $F$ 
and discuss the combinatorial structure of $F$.

\subsection{The combinatorial structure of $F$.}
Recall that the Feigenbaum map $F$ is a solution of  Cvitanovi\'c-Feigenbaum equation:
\begin{equation}\label{EqCvitFeig}
\left\{
\begin{array}{lll}F(z) & = & -\tfrac{1}{\lambda}F^2(\lambda z),\\
F(0) & = & 1,\\ F(z) & = & H(z^2),\;\text{with}\;H^{-1}(z)\;\text{univalent in}
\;\mc_\lambda,
\end{array}\right.
\end{equation} where $\mc_\lambda:=\mc\setminus((-\infty,-\tfrac{1}{\lambda}]\cup[\tfrac{1}{\lambda^2},\infty))$
 and $\tfrac{1}{\lambda}=2.5029\ldots$ is one of the Feigenbaum constants. From (\ref{EqCvitFeig}) we
 immediately obtain:
 \begin{equation}\label{EqF2m}
 F^{2^m}(z)=(-\lambda)^mF(\tfrac{z}{\lambda^m})
 \end{equation} whenever both sides of the equation are defined.
 Another corollary of \eqref{EqCvitFeig} is the following (cf. Epstein \cite{Ep_l}):
\begin{Prop}\label{PropX0} Let $x_0$ be the first positive preimage of $0$ by $F$.
Then \[F(\lambda x_0)=x_0,\; F(1)=-\lambda,\;F(\tfrac{x_0}{\lambda})=-\tfrac{1}{\lambda}\] and $\tfrac{x_0}{\lambda}$ is the first positive critical point of $F$.
\end{Prop}
 A map $g:U_g\to\mc$ is called an analytic extension of a map $f:U_f\to \mc$ if
 $f$ and $g$ are equal on some open set. An extension $\hat f:S\supset U\to\mc$ of $f$ is called a maximal analytic extension if every analytic extension of $f$ is a restriction of $\hat f$. The following crucial observation is also due to H.~Epstein (cf. 
\cite{Ep_l,Bu}): 
\begin{Th} The map $F$ has a maximal analytic extension 
$\hat F:\hat W\to\mc,$ where $\hat W\supset \mr$ is an open simply connected set.
\end{Th}
\noindent 
For simplicity of notation, in what follows we will routinely identify $F$ with its maximal analytic extension $\hat F$.

Set $\mh_+=\{z:\im z>0\}$ and $\mh_-=\{z:\im z<0\}$. For a proof of the following, see \cite{Ep_l,Bu}:
\begin{Th}\label{ThCrit} All critical points of $F$ are simple. The critical values of $F$ are contained in real axis. Moreover, for any $z\in \hat W$ such that $F(z)\notin\mr$ there exists a bounded open set $U(z)\ni z$ such that $F$ is one-to-one on $U(z)$ and $F(U(z))=\mh_\pm$.
\end{Th}
\label{th-tiles}
We illustrate the statement of the theorem in Figure \ref{fig-tiles} (a very similar figure appears in X.~Buff's paper \cite{Bu}). The lighter and darker ``tiles'' are the bounded connected components of the preimage of $\mh_+$ and $\mh_-$ respectively. Black tree is the boundary of the domain $\hat W$.  In the gray region, colors cannot be effectively rendered at the given resolution.

 \begin{figure}\centering\includegraphics[width=0.95\textwidth]{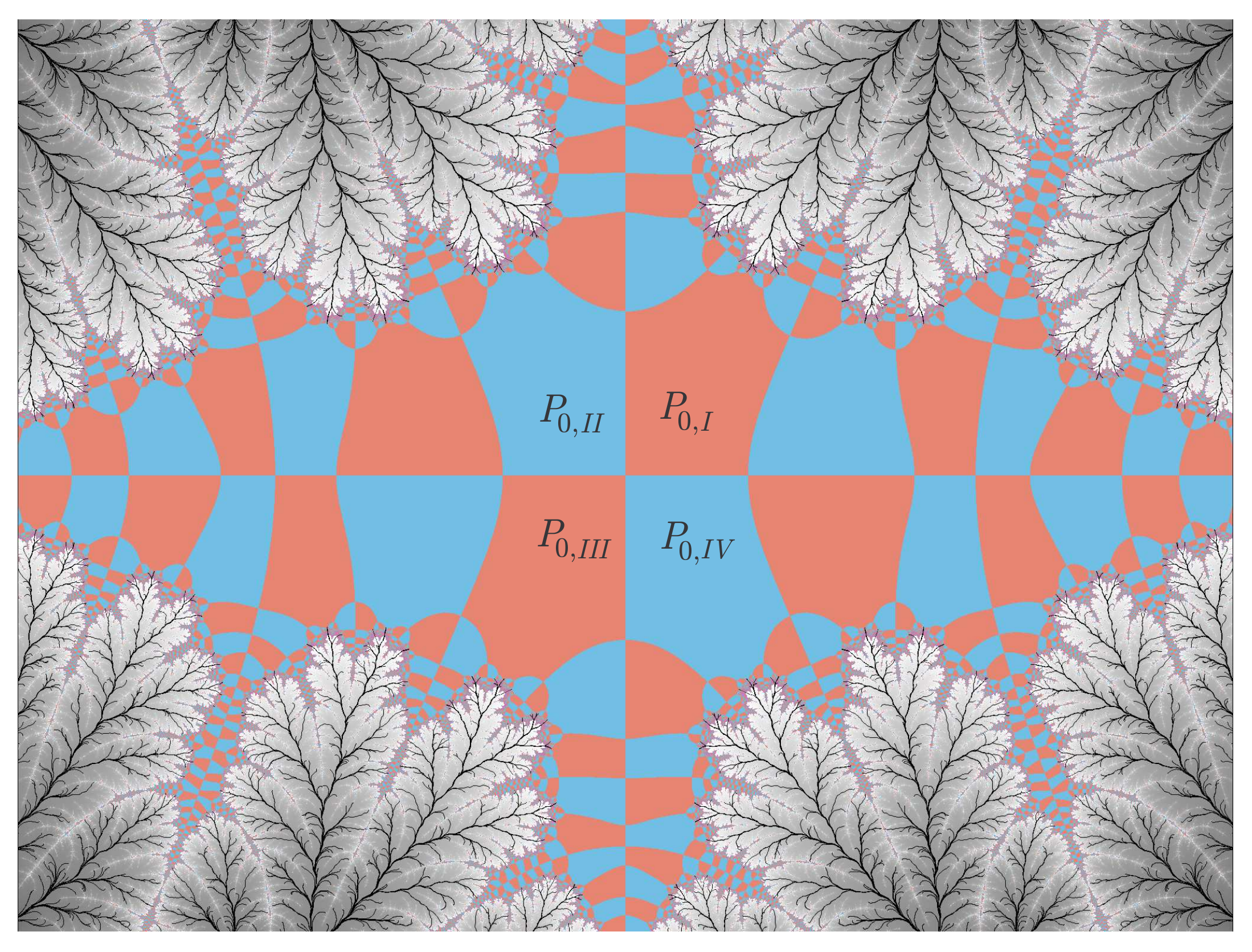}\caption{\label{fig-tiles}Illustration to Theorem \ref{th-tiles}. We thank Scott Sutherland for computing this image for us. } \end{figure} 

Following \cite{Bu}, we introduce the following combinatorial partition of $\hat W$.
\begin{Def} Denote by $\cp$ the set of all connected components of $F^{-1}(\mc\setminus \mr)$. Set\[\cp^{(n)}=\{\lambda^nP:P\in\cp.\}\]
\end{Def}
Thus, for any $P\in\cp$ the map $F$ sends $P$ one-to-one either onto $\mh_+$ or onto $\mh_-$. Notice that $\cp$ is invariant under multiplication by $-1$ and under complex conjugation. Using 
Cvitanovi\'c-Feigenbaum equation we obtain the following:
\begin{Lm}\label{LmPartition} For any $n$, the partition $\cp^{(n)}$ coincides with the set of connected componets of the preimage
of $\mc\setminus\mr$ under $F^{2^n}$. Moreover, one has:
\begin{itemize}
\item[1)] for any $m\in \mn\cup\{0\}$, $m>n$ and any $P\in \cp^{n}$ the iterate $F^{2^n-2^m}$ maps $P$ bijectively onto some $Q\in \cp^{(m)}$;
\item[2)] for any $n\in\mn,s\in\mn,s\le 2^{n-1}$ and any $P\in \cp^{(n)}$ there exists $Q_0\in\cp^{(n)},Q_1\in\cp^{(n-1)}$ such that $Q_0\subset F^s(P)\subset Q_1$.
\end{itemize}
\end{Lm}
%\begin{proof} By induction, one has: \[F^{2^n}(z)=(-\lambda)^nF(\tfrac{z}{\lambda^n}).\] Let $P\in\cp^{(n)}$. Then $P=\lambda^nQ$ for some $Q\in \cp$. Since $F:Q\to \mh_\pm$ is one-to-one and onto, $F^{2^n}:P\to \mh_\pm$ is one-to-one and onto. $1)$ follows immediately.
%
%To prove $2)$ set $Q=f^s(P)$. Since $2^{n-1}\le 2^n-s$ the map $f^{2^{n-1}}$ is one-to-one on $Q$. Clearly,
%$f^{2^{n-1}}(Q)\subset\mc\setminus\mr$. Therefor, $Q$ lies inside some connected component $Q_1$ of the preimage of $\mc\setminus\mr$ by $f^{2^{n-1}}$. To find $Q_0$ without loss of generality we may assume that $s=1$.
%
%
%\end{proof}
Let us describe the structure of $F$ on the real line near the origin. Since $F$ maps $[1,\tfrac{x_0}{\lambda}]$ homeomorphically onto $[-\tfrac{1}{\lambda},-\lambda]$ there exists a unique $a\in (1,\tfrac{x_0}{\lambda})$ such that $F(a)=-\tfrac{x_0}{\lambda}$.
\begin{Lm}\label{LmCritPoints} The first three positive critical points of $F$ counting from the origin are $\tfrac{x_0}{\lambda},\tfrac{a}{\lambda},\tfrac{x_0}{\lambda^2}$. One has:\[F(\tfrac{x_0}{\lambda})=-\tfrac{1}{\lambda},\;F(\tfrac{a}{\lambda})=\tfrac{1}{\lambda^2}.\]
\end{Lm}
\begin{proof} By Cvitanovi\'c-Feigenbaum equation, $F'(z)=-F'(\lambda z)F'(F(\lambda z))$. One has:
\begin{align*}
%F'(\tfrac{x_0}{\lambda})=-F'(x_0)F'(F(x_0))=-F'(x_0)F'(0)=0,\\
F'(\tfrac{a}{\lambda})=-F'(a)F'(F(a))=-F'(a)F'(-\tfrac{x_0}{\lambda})=0,\\
F'(\tfrac{x_0}{\lambda^2})=-F'(\tfrac{x_0}{\lambda})F'(F(\tfrac{x_0}{\lambda}))=0.
\end{align*}
Assume that there is another critical point on $(0,\tfrac{x_0}{\lambda^2})$. Let $b$ be the minimal such  critical point. Then \[F'(b)=0,F'(\lambda b)\neq 0\;\text{therefore}\;F'(F(\lambda b))=0.\] Since $\lambda b\in (0,\tfrac{x_0}{\lambda})$, using Proposition \ref{PropX0} we get that $F(\lambda b)\in (-\tfrac{1}{\lambda},1)$. Then one of the following three possibilities holds:
\begin{itemize}
\item{}$F(\lambda b)=0\;\;\Rightarrow\;\;\lambda b=x_0,\;b=\tfrac{x_0}{\lambda}$;
\item{}$F(\lambda b)=-\tfrac{x_0}{\lambda}\;\;\Rightarrow\;\;\lambda b=a,\;b=\tfrac{a}{\lambda}$;
\item{}$F(\lambda b)\in (-\tfrac{1}{\lambda},-\tfrac{x_0}{\lambda})\;\;\Rightarrow\;\;\lambda b\in (a,\tfrac{x_0}{\lambda})$ and hence $b>\tfrac{a}{\lambda}>\tfrac{1}{\lambda}>-F(\lambda b)$.
\end{itemize}
Each of the possibilities above contradicts  the choice of $b$.
\end{proof}
Now, Theorem \ref{ThCrit} together with Lemma \ref{LmCritPoints} imply that for each of the segments
 \[[0,\tfrac{x_0}{\lambda}],[\tfrac{x_0}{\lambda},\tfrac{a}{\lambda}],[\tfrac{a}{\lambda},\tfrac{x_0}{\lambda^2}]\] there exists exactly one tile $P\in\cp$ in the first quadrant which contain this segment in its boundary. Denote these tiles by $P_{0,I},P_{1,I}$ and $P_{2,I}$ correspondingly. For a quadrant $J\neq I$ (that is, $J=II,III$ or $IV$) and $k\in\{0,1,2\}$ let $P_{k,J}\in P$ be the tile
 in quadrant $J$ which is symmetric to $P_{k,I}$ with respect to one of the axis or the origin. For any set $P$ and any $n$ set $P^{(n)}=\lambda^n P$.
 \begin{figure}\centering\includegraphics[width=0.85\textwidth]{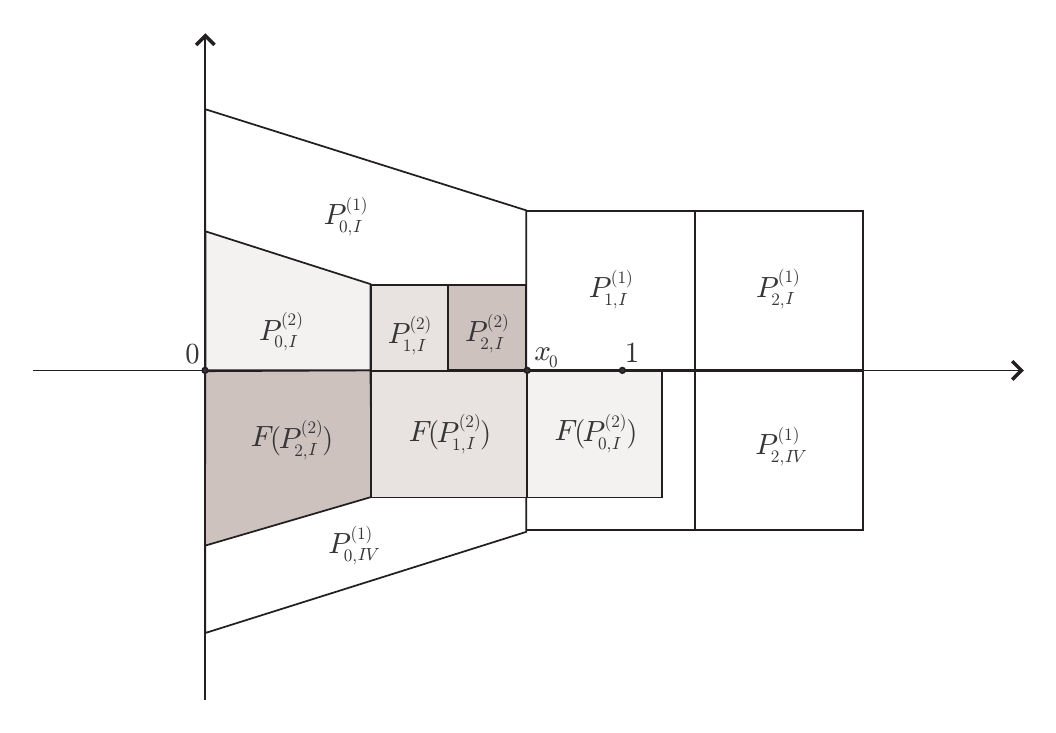}\caption{Illustration to Proposition \ref{PropTilingStr}.} \end{figure} 
 \begin{Prop}\label{PropTilingStr} The Feigenbaum map $F$ satisfies the following:
 \begin{itemize}
 \item[1)] $F(P_{0,I}^{(2)})\subset P_{1,IV}^{(1)}$;
 \item[2)] $P_{1,IV}^{(2)}\cup P_{2,IV}^{(2)}\subset F(P_{1,I}^{(2)})\subset P_{0,IV}^{(1)}$;
 \item[3)] $F(P_{2,I}^{(2)})\subset P_{0,IV}^{(1)}$;
 \item[4)] $F(P_{0,I}^{(1)})=P_{0,IV}^{(0)}$, $F(P_{1,I}^{(1)})=P_{0,III}^{(0)}$.
 \end{itemize}
 \end{Prop}
 \begin{proof} Since $$P(0)=1\in P_{1,IV}^{(1)}\text{ and }F(P_{0,I}^{(2)})\subset F(P_{0,I}^{(0)})=\mh_-,$$ by Lemma \ref{LmPartition} $2)$ we get that $F(P_{0,I}^{(2)})\subset P_{1,IV}^{(1)}$.

Further, we have:
\[F(\lambda x_0)=x_0,\;\;\text{and}\;\;F(F(\lambda a))=-\lambda F(a)=x_0.\] Since $F$ is one-to-one on $[0,\tfrac{x_0}{\lambda}]$, it follows that $F(\lambda a)=\lambda x_0$. Thus,
\[F((\lambda x_0,\lambda a))=(\lambda x_0,x_0).\] By Lemma \ref{LmPartition} $2)$ we obtain the property $2)$ of Proposition \ref{PropTilingStr}. The properties $3)$ and $4)$ can be proven in a similar fashion.
\end{proof}

\subsection{Computing the Feigenbaum map $F$.}
Let us set
 \[W=\text{Int}\overline{P_{0,I}\cup P_{0,II}\cup P_{0,III}\cup P_{0,IV}}.\] 
 Consistently with our previous notation, let us  define $W^{(0)}=W$ and $W^{(n)}=\lambda^nW$. 

Let us fix a rational number $r>0$ and 
 a dyadic set $U$ such that 
\begin{equation}\label{EqSetU}U_r(W^{(1)})\subset U\;\;\text{and}\;\;U_r(U)\subset W^{(0)}.
\end{equation}

We state:
\begin{Prop} \label{prop:feigpoly}
The restriction of $F$ onto $U$ is poly-time computable.
\end{Prop}
The proof of Proposition \ref{prop:feigpoly} will occupy the rest of the section. 

Let us begin by defining some functional spaces.
For a topological disk $W\subset\mc$  we will denote 
$\ca_W$ the Banach space of bounded analytic functions in $W$ equipped with the
sup norm. 
In the case when the domain $W$ is the disk $\md_\rho$ of radius $\rho>0$ centered
at the origin, we will denote $\ca_{\md_\rho}\equiv \ca_\rho$.

For each $\rho>0$ we will also consider the 
collection $\cl^1_\rho$ of analytic functions $f(z)$ defined on $\md_\rho$, equipped with the weighted $l_1$
norm on the coefficients of the Maclaurin's series:

 \begin{equation}
 \|f\|_\rho=\sum^{\infty}_{n=0}{\left| f^{(n)}(0) \right| \over n !} \rho^n.
 \end{equation}

\noindent
The proof of the following elementary statement is left to the reader:

\begin{Lm}\label{Equivalence}
$ \ $  
\begin{itemize}
\item[1)]  Let $f \in \cl^1_\rho$, then 
$\sup_{\md_\rho} |f(z)| \le  \|f\|_\rho$;
\item[2)]  Let $f \in \ca_{\rho'}$ and $\rho'>\rho$, then $ \|f\|_{\rho} \le {\rho \over \rho'-\rho}  \sup_{\md_{\rho'}} |f(z)|.$
\end{itemize}
\end{Lm}

As an immediate consequence, we have:

\begin{Co}
$\cl_\rho^1$ is a Banach space.
\end{Co}

\noindent
To compute the Feigenbaum map $F$ we will recall a rigorous computer-assisted
 approach of Lanford \cite{Lan}, based on an approximate Newton's method
for $\cR_2$. Note that Lanford also proved hyperbolicity of period-doubling renormalization at $F$ (but not uniqueness of $F$)
using the same approach before Sullivan's work which we have quoted above.

Lanford used the Contraction Mapping Principle to find $F$. Since $\cR_2$ is not a contraction -- as it has an unstable eigenvalue
at $F$ -- he replaced the fixed point problem for $\cR_2$ with the fixed point problem for the approximate Newton's Method
$$g\mapsto g+(I-\Gamma)^{-1}(\cR_2(g)-g),$$
where $\Gamma$ is a high-precision finite approximation of $D\cR_2|_F.$
Formally, his results can be summarized as follows:

\begin{Th}\cite{Lan}
There exist rational numbers $\rho>0$ and $\Delta>0$, a polynomial $p(z)$ with
 rational coefficients, and an explicit linear operator $\Gamma$ on $\cl^1_\rho$ (which is given by a finite rational matrix in the
canonical basis of $\cl^1_\rho$)  such that the following properties hold.

The operator $I-\Gamma$ is invertible.
Denoting $\cd$ the Banach ball in $\cl^1_\rho$ given by
$$||g-p||_\rho<\Delta,$$
and 
$$\Phi:g\mapsto g+(I-\Gamma)^{-1}(\cR_2(g)-g),$$
we have:
\begin{itemize}
\item $\Phi(\cd)\Subset\cd$;
\item moreover, there exists $\rho'>\rho$ such that for every $g\in\cd$ the image $\Phi(g)\in \cl^1_{\rho'}$;
\item the Feigenbaum map $F\in \cd$ (note that it immediately follows that $\Phi(F)=F$);
\item finally, there exists a positive $\epsilon <1$ such that
$$||\Phi(g)-F||_\rho<\epsilon||g-F||_\rho $$
for all $g\in\cd$.
\end{itemize}

\end{Th}

Note that by Cvitanovi\'c-Feigenbaum equation (\ref{EqCvitFeig}) to prove Proposition \ref{prop:feigpoly}, it is
sufficient to compute $F$ in polynomial time in a disk $\cd_r$ for some $r>0$. Let $\rho$ and $\rho'$ be as above.
Fix $m$, where $2^{-m}$ is the desired precision for $F$. 
For a decimal number $b$ denote $\lfloor b\rfloor_l$ its round-off to the $l$-th decimal digit.
For 
$$f=\sum_{k=1}^\infty b_kz^k\text{ set }\lfloor f\rfloor_l\equiv\sum_{k=1}^\infty \lfloor b_k\rfloor_l z^k.$$
Fix $l=O(m)$ such that for all $g\in\cd$ 
$$||g-\lfloor g\rfloor_l||_\rho<2^{-(m+2)}. $$
Further, for 
$$f=\sum_{k=1}^\infty b_kz^k\text{ set }\text{Poly}_n(f)\equiv\sum_{k=1}^n b_k z^k.$$
Applying Cauchy derivative estimate to the remainder term in Maclaurin series, we see that 
there exists $n=O(m)$ such that for all $g\in\cd\cap\cl^1_{\rho'}$
$$||g-\text{Poly}_n( g)||_\rho<2^{-(m+2)}. $$
Now let $$p_0=\lfloor p_0\rfloor_l\sum_{n=0}^n a_k^0 z^k\in\cd$$ be a polynomial with rational coefficients. 
The binomial formula implies that 
$\Phi(p_0)$ can be computed in time $O(m^4)$. Define 
$$p_1=\text{Poly}_n(\lfloor \Phi(p_0)\rfloor_l).$$
Note that
$$||p_1-F||_\rho<\epsilon||p_0-F||_\rho+2^{-(m+1)}.$$
Iterating the procedure $O(m)$ times (total computing time $O(m^5)$), we obtain a $2^{-(m+1)}$-approximation of  $F$ in $||\cdot||_\rho$.
This, and Lemma \ref{Equivalence} imply the desired statement.
%%%%%%%%%%%%%%%%%%%%%%%%%%%%%%%%%%
%%%%%%%%%%%%%%%%%%%%%%%%%%%%%%%%%%
\section{Computing long iterations.}\label{SecCompIter}
  Introduce the following notations:
  \[Q_J=\text{Int}\overline{P_{0,J}\cup P_{1,J}},\;\;R_J=\text{Int}\overline{P_{1,J}\cup P_{2,J}},\;\;S_J=\text{Int}\overline{P_{0,J}\cup P_{1,J}\cup P_{2,J}},\]
  where $J\in\{I,II,II,IV\}$ and $\text{Int}$ stands for the interior of a set. Let
 \[P_k=\text{Int}\overline{P_{k,I}\cup P_{k,II}\cup P_{k,III}\cup P_{k,IV}}.\] Similarly define $Q,R$ and $S$. 
 Note that $W^{(0)}=P_{0}$. Recall that for a set $P$ we defined $P^{(n)}=\lambda^nP$.

Observe that $J_f\Subset Q^{(1)}\Subset W^{(0)}=W.$ For any $w\in W$ let $m(w)$ be such that $w\in W^{(m)}\setminus W^{(m+1)}$.
Fix a point $z_0=z\in Q^{(1)}$. Introduce inductively a sequence $\{z_k\}$ of iterates of $z$ under $F$ as follows:
\begin{equation}\label{EqZkDef}
z_{k+1}=\left\{\begin{array}{ll}F(z_k),&\text{if}\;\;z_k\in  Q^{(1)}\setminus W^{(1)},\\
F^{2^{m_k-1}}(z_k),&\text{if}\;\;m_k=m(z_k)\ge 1.\end{array}\right.
\end{equation} 
If $z_k\notin  Q^{(1)}$ then the sequence terminates at the index $k_{term}:=k$. For every $k$ let $m_k$ be the number
 such that $z_k\in W^{(m_k)}\setminus W^{(m_k+1)}$.
By \eqref{EqF2m} if $m_k\ge 1$ then one has:
\[z_{k+1}=(-\lambda)^{(m_k-1)}F(z_k/\lambda^{m_k-1}).\] In particular, this implies that 
\begin{equation}\label{EqMkgrowth}m_{k+1}\ge m_k-1\;\;\text{for all}\;\;k.
\end{equation}

Define inductively a sequence of indexes $s_k$ such that $z_k=f^{s_k}(z)$:
\begin{equation}\label{EqSk}s_0=0,\;\;s_{k+1}=\left\{\begin{array}{ll}s_k+1,&\text{if}\;\;m_k=0,\\
s_k+2^{m_k-1},&\text{if}\;\;m_k\ge 1.\end{array}\right.
\end{equation}
%From Proposition \ref{PropTilingStr} it follows
 %that $\big(\overline{P_{I,0}^{(0)}}\setminus \overline {Q_I^{(1)}}\big)\cap %\mr=\varnothing$. 
 Set
 \[\epsilon=\lambda d\big(W\setminus S^{(1)},\mr\big).\] If there exists 
$i$ such that $z_i\in W^{(1)}$ then denote $j_{term}=\max\Big\{i:z_i\in W^{(1)}\Big\}$. Otherwise set $j_{term}=\infty$.

The main result of this section is the following:
\begin{Prop}\label{PropPolyEsc} There exist constants $A,B>0$ such that if $d(z,J_f)\ge 2^{-n}$ then the sequence $\{z_k\}$ terminates at some index $k=k_{term}\le An+B$.
\end{Prop}
\noindent
The rest of Section \ref{SecCompIter} is devoted to the proof of Proposition \ref{PropPolyEsc}.
But first let us state an important
\begin{Co}\label{CoPolyEsc} If $d(z,J_f)\ge 2^{-n}$ then $m(F^s(z))\le An+B$ for all $s\le s_{k_{term}}$.
\end{Co}
\begin{proof} Proposition \ref{PropPolyEsc} and \eqref{EqMkgrowth} imply that $m_j\le An+B$ for all $j$. Let 
$s_j<s<s_{j+1}$. Then $$m_j>1\;\;\text{and}\;\;s_{j+1}=s_j+2^{m_j-1}.$$ We have: $F^{s_j}(z)=z_j\in W^{(m_j)}.$ Since the first landing map
 from $W^{(m_j)}$ to $W^{(m_j-1)}$ is $F^{2^{m_j-1}}=F^{s_{j+1}-s_j}$ we get $$m(F^s(z))\le m_j-1,$$ which finishes the proof.
\end{proof}

\begin{Lm}\label{LmZkFin} Let $z\in W^{(0)}\setminus J_F$, then $\{z_k\}$ is finite. Moreover, if $j_{term}<\infty$  then $z_{j_{term}}\notin S^{(2)}$ and hence $|\im(z_{j_{term}})|>\epsilon$.
\end{Lm}
\begin{proof} Let $z\in  Q^{(1)}\setminus J_F$. Assume that $\{z_k\}$ is infinite. Then $z_k\in  Q^{(1)}$ for every $k$. Since $F_{| Q^{(1)}}$ is quadratic-like, there exists $l$ such that $F^l(z)\notin  Q^{(1)}$. Let $k$ be the maximal index such that 
$s_k<l$. Two cases possible:
\\$a)$ $m_k=0$ or $m_k=1$. Then $z_{k+1}=F(z_k)$ and $s_{k+1}=s_k+1$. It follows that 
$l=s_{k+1}=s_k+1$ and $F^l(z)=z_{k+1}\in  Q^{(1)}$. We arrive at a contradiction.\\
$b)$ $m_k\ge 2$. Since $z_k\notin J_f\supset \mr\cup \i\mr$ we obtain that $z_k\in P_{0,J}^{(m_k)}$ 
for some $J$. 
Observe that $s_{k+1}=s_k+2^{m_k-1}\ge l$. Lemma \ref{LmPartition} implies that 
\[F^{l-s_k}( P_{0,J}^{(m_k)})\subset T,\;\;\text{where}\;\;T\in \cp^{(m_k-1)}.\] Clearly, $T\cap J_F\neq\varnothing$.
Thus, $F^l(z)$ belongs to a tile $T$ of level $m_k-1\ge 1$ which intersects $J_F$. This implies that 
$F^l(z)\in Q^{(1)}$. 
We arrive at a contradiction. This shows that the sequence $z_j$ is finite.

Further, assume that $j=j_{term}<\infty$. Set $k=k_{term}$. Observe that $z_k\notin  Q^{(1)}$. It follows from Proposition \ref{PropTilingStr} that $z_{k-1}
\notin S^{(2)}$. Thus, if $z_{k-1}\in W^{(1)}$ then $j=k-1$ and 
\[|\im(z_j)|\ge d(W^{(1)}\setminus S^{(2)},\mr)=\epsilon.\] 
Otherwise, $k\ge j+2,z_{j+1}\in Q^{(1)}\setminus W^{(1)}$ and 
$z_{j+2}\notin W^{(1)}$. Since 
\[F^2(W^{(2)})\subset W^{(1)},\;\;F\big(P_1^{(2)}\big)\subset W^{(1)}\;\;
\text{and}\;\;F\big(P_2^{(2)}\big)\subset W^{(1)}\] we obtain that  $z_j\notin 
S^{(2)}$. It follows that 
$|\im z_j|\ge \epsilon$.
\end{proof}

\subsection{Expansion in the hyperbolic metrics.}
Fix $z_0\in  Q^{(1)}\setminus J_f$. Let $\{z_k\}$ be the sequence defined above. Set $r_k=\max\{m_k-1,0\}$, so that $z_{k+1}=F^{2^{r_k}}(z_k)$ for all $k<k_{term}$. Introduce an auxiliary sequence $w_k=z_k/\lambda^{r_k}$. From \eqref{EqF2m} we obtain that 
\[w_{k+1}=(-1)^{r_k}\lambda^{r_k-r_{k+1}}F(w_k).\] Observe that $w_k\subset Q^{(1)}\setminus W^{(2)}$. Therefore,
 $F(w_k)\subset W$. It follows that for all $k<k_{term}$ we have:
$r_{k+1}\ge r_k-1$. For convenience, set \[H_k(w)=
\lambda^{r_k-r_{k+1}}F(w),\;\;H_{k,l}=H_{l-1}\circ H_{l-2}\circ\ldots\circ H_{k+1}\circ H_k,\;\;k<l\] so that $w_{k+1}=\pm H_k(w_k),\;\;w_l=\pm H_{k,l}(w_k),\;k<l$.
 For a point $z$ such that $F(z)\notin \mr$ define 
 by $\|DF(z)\|_{\mh}$ the norm of the differential of $z$ in the hyperbolic metrics on either $\mh_+=\{z:\im z>0\}$ or $\mh_-=\{z:\im z<0\}$. Since $F$ is even and one-to-one from $P_{0,I}^{(0)}$ onto $\mh_+$, from Schwarz-Pick Theorem we obtain the following:
 \begin{Lm}\label{LmHypExp} For all $w\in W\setminus (\mr\cup\i\mr)$ one has $\|D F(w)\|_\mh>1$. Moreover, there exists $\lambda_1>1$ such that $\|DF(w)\|_\mh >\lambda_1$, assuming that $w\in W^{(1)}\setminus  S^{(2)}$.
 \end{Lm}
 Let $N=N(z)$ be the number of indexes $k$ for which $m_k\ge 1$ and $z_k\notin S^{(m_k+1)}$. Since the hyperbolic metric on $\mh_\pm$ is scaling invariant, using \eqref{EqF2m} we get:
 \begin{Prop}\label{PropHypExp} If $m_k\ge 1$ and $z_k\notin S^{(m_k+1)}\cup \i\mr$, then 
 \[\big\|DF^{2^{m_k-1}}(z_k)\big\|_\mh>\lambda_1.\] Moreover, there is a universal constant $C_1$ (independent from $z$) such that
 \[\big\|D F^{s_{j_{term}}}(z)\big\|_\mh>\lambda_1^{N-1},\;\;\big|DF^{s_{j_{term}}}(z)\big|>C_1\lambda_1^N,\] assuming 
 that $j_{term}<\infty$.
\end{Prop}
\begin{proof} If $z_k\notin S^{(m_k+1)}$ then $z_k/\lambda^{m_k-1}\notin S^{(2)}$. By Lemma \ref{LmHypExp} we obtain:
\[\big\|DF^{2^{m_k-1}}(z_k)\big\|_\mh=\big\|DF(z_k/\lambda^{m_k-1})\big\|_\mh>\lambda_1.\] It follows that 
$\big\|DF^{s_{j_{term}}}(z)\big\|_\mh>\lambda_1^{N-1}$. Since the hyperbolic metric of $\mh_\pm$ is equivalent to the Euclidean metric on any compact subset of $\mh_\pm$, using Lemma \ref{LmZkFin} we obtain the last inequality of Proposition \ref{PropHypExp}.
\end{proof}

Set \begin{align*}R_+=\inter\overline{R_I\cup R_{IV}}=R\cap \{z:\re z>0\},
\\ W_+=\inter\overline{P_{0,I}\cup P_{0,IV}}=
W\cap \{z:\re z>0\},\;\;F_+=F_{|W_+}.\end{align*}
Observe that \[F(W_+)=\mc\setminus((-\infty,-\tfrac{1}{\lambda}]\cup[1,+\infty))\Supset W_+^{(1)}\supset R^{(2)}_+.\] 

\begin{figure}\begin{center}\includegraphics[width=0.5\textwidth]{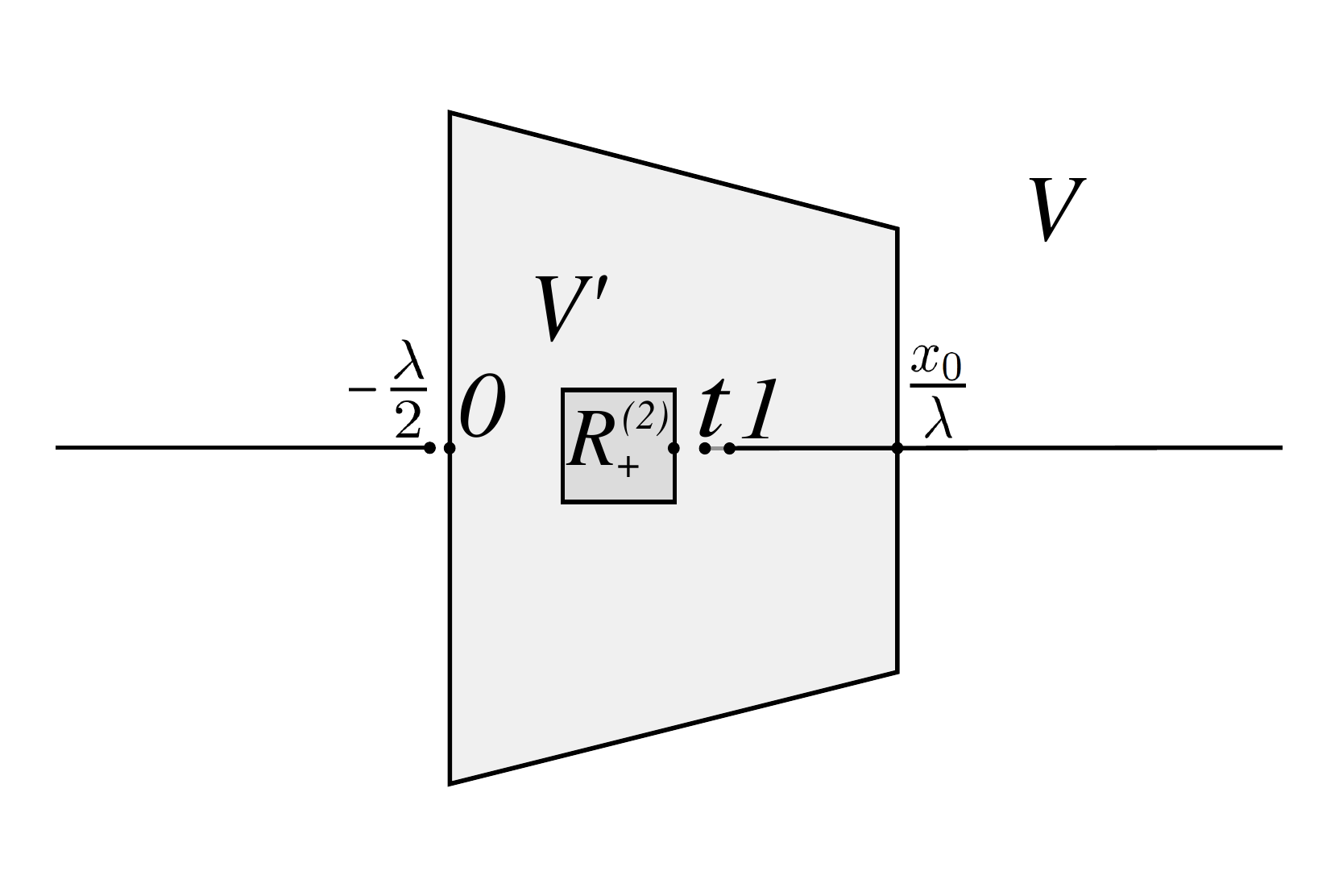}\caption{The sets $V,V'$ and $P_+^{(2)}$.} \end{center}\end{figure}
Introduce sets 
\[V=\mc\setminus((-\infty,-\tfrac{\lambda}{2}]\cup [1,+\infty)),\;\;V'=F_+^{-1}(V)=W_+\setminus (t,\tfrac{x_0}{\lambda}),\]
where $t=F_+^{-1}(-\tfrac{\lambda}{2})\subset (x_0,1),$ since $F_+^{-1}(-\lambda)=1$ and $F_+^{-1}(0)=x_0$. Notice that \[R_+^{(2)}\Subset V'\subset V.\] 
\begin{Rem} Let $k$ be an index such that $w_k\in R^{(2)}_+$ and $m_k\ge 1$. Then, by definition, 
\[H_k(w_k)=\lambda^{r_k-r_{k+1}}F(w_k)=\pm w_{k+1}\subset W_+^{(1)}\Subset V.\] Therefore,
 the norm of the derivative $DH_k(w_k)$ in the hyperbolic metric on $V$ is well defined.
\end{Rem}
\begin{Prop}\label{PropExponV} There exists $\lambda_2>1$ such that if $w_k\in R^{(2)}_+$ and $m_k\ge 1$ then \[\|DH_k(w_k)\|_V\ge \lambda_2.\]
\end{Prop}
\noindent
Let us introduce an auxiliary function 
\[a(r)=\frac{1-x(r)^2}{2|x(r)\log x(r)|},\;\;\text{where}\;\;x(r)=\frac{e^r-1}{e^r+1}.\] Observe that 
 $a(r)$ is decreasing on $[0,\infty)$, $a(r)\to \infty$ when $r\to 0+$ and $a(r)\to 1$ when $r\to\infty$.
The proof of Proposition \ref{PropExponV} relies on the following consequence of the Schwarz-Pick Theorem:
\begin{Lm} Let $U\subset V$ be domains in $\mathbb{C}$, $G:U\to V$ be a conformal map, $z\in U$ and 
$r=\d_V(z,V\setminus U)$. Then $\|D G(z)\|_V\ge a(r)$.
\end{Lm}
\begin{proof}
Since $\|DG(z)\|_{U,V}=1$ we obtain:
\[\|DG(z)\|_V=\|D\id(z)\|_{V,U},\] where 
$\id$ is the identity map. Let $\zeta\in V\setminus U$ and 
$R=\d_V(\zeta,z)$. Set $\widetilde V=V\setminus\{\zeta\}$. By the Schwarz-Pick Theorem, 
\[\|D\id(z)\|_{V,U}=\|D\id(z)\|_{V,\widetilde V}\|D\id(z)\|_{\widetilde V,U}\ge \|D\id(z)\|_{V,\widetilde V}.\] Let $\phi:V\to \mathbb U$ be the conformal map such that $\phi(\zeta)=0$ and 
$w=\phi(z)>0$.  Then
 \[\|D\id(z)\|_{V,\widetilde V}=\|D\id(w)\|_{\mathbb U,\mathbb U\setminus\{0\}},\;\;\text{and}\;\;
 R=\d_V(z,\zeta)=\d_{\mathbb U}(w,0).\] The value of  $\|D\id(w)\|_{\mathbb U,\mathbb U\setminus\{0\}}$ can be computed explicitly and is equal to $a(R)$. Since $\zeta$ is any point in $V\setminus U$ we obtain that $\|D G(z)\|\ge a(r)$.
\end{proof}

\begin{proof}[Proof of Proposition \ref{PropExponV}] Let $w_k\in R^{(2)}_+$. 
Then \[r_{k+1}\ge r_k\;\;\text{and}\;\;\lambda^{r_{k+1}-r_k}V\subset V.\] Set  
$A=F_+^{-1}(\lambda^{r_{k+1}-r_k}V)$. Then $H_k:A\to V$ is a conformal isomorphism and $w_k\in A$. Since 
\[R_+^{(2)}\Subset V'\subset V\;\;\text{and}\;\;A\subset V'\] we obtain that $R=\d_V(R_+^{(2)},V\setminus V')+\diam_V(R_+^{(2)})$ is finite and \[\d_V(w_k,V\setminus A)\le R.\] By Lemma \ref{LmHypExp} we obtain that $\|DH_k(w_k)\|_V\ge a(R)$.
\end{proof}
\begin{Prop}\label{PropEuclExpV} There exists $1>C_2>0$ such that the following is true. Let $k,l$ be such 
that $w_j\in R^{(2)}$  for $j=k,k+1,\ldots,k+l-1$. Then
\[\big|DH_{k,k+l}(w_k)\big|\ge C_2\lambda_2^l.\]
\end{Prop}
\begin{proof} Let $k,l$ be as in the conditions of the proposition. Without loss of generality we may assume that 
$w_k\in R^{(2)}_+$. Let $\rho(z)\dd z$ be the hyperbolic metric on $V$. Since $R_+^{(2)}\Subset V$ there exists a constant $M>0$ such that 
\[\tfrac{1}{M}<|\rho(z)|<M\;\;\text{and}\;\;|DF(z)|>\tfrac{1}{M}\;\;\text{for all}\;\;z\in R_+^{(2)}.\] 
Notice that for all $k\le j\le k+l-1$ we have:
\[F(w_j)\subset W_+^{(1)},\;\;\lambda^{r_j-r_{j+1}}F(w_j)=H_j(w_j)=\pm w_{j+1}\subset W_+^{(1)}\setminus W_+^{(2)},\] therefore 
$r_j-r_{j+1}\le 0$ and $|DH_j(z)|=\lambda^{r_j-r_{j+1}}|DF(z)|\ge \tfrac{1}{M}$ for all $z\in R_+^{(2)}$.
By Proposition \ref{PropExponV} we obtain:
\[|DH_{k,k+l}(w_k)|=|DH_{k+l-1}(w_{k+l-1})|\cdot|DH_{k,k+l-1}(w_k)|\ge \tfrac{1}{M^3}\|DH_{k,k+l-1}(w_k)\|_V\ge
\tfrac{\lambda_2^{l-1}}{M^3},\] which finishes the proof.
\end{proof}
Further, set \[P_{1,+}=\inter\overline{P_{1,I}\cup P_{1,IV}},\;W_-=-W_+.\] Then $F:P_{1,+}^{(1)}\to W_-$
is a conformal isomorphism. Notice that $P_{1,+}^{(1)}\Subset W_+$. Similarly to Proposition \ref{PropEuclExpV} we obtain:
\begin{Prop}\label{PropEuclExpP1} There exists $\lambda_3>1,1>C_3>0$ such that the following is true. Let $k,l$ be such 
that $w_j\in P_1^{(1)}$ for $j=k,k+1,\ldots,k+l-1$. Then
\[\big|DH_{k,k+l}(w_k)\big|\ge C_3\lambda_3^l.\]
\end{Prop}
\subsection{Proof of Proposition \ref{PropPolyEsc}.} 
\begin{Lm}\label{LmUsingKoebe} There exists a constant $C>1$ such that if one of the following is true\\
$1)$ $F^j(z)\in W$ for $j=0,1,\ldots,s-1$ and $F^s(z)\in W^{(1)}\setminus S^{(2)}$,\\
$2)$ $z\in Q^{(1)}\setminus W^{(2)}$, and 
$F^j(z)\in Q^{(1)}\setminus W^{(1)}$ for $j=1,\ldots,s-1$ then 
\[d(z,J_f)\le \frac{Cd(F^s(z),J_f)}{|DF^s(z)|}.\]
\end{Lm}
\begin{proof} Assume that the first condition is true: $F^j(z)\in W$ for $j=0,1,\ldots,s-1$ and $F^s(z)\in W^{(1)}\setminus S^{(2)}$. Without loss of generality let $\im F^s(z)>0$. Fix two convex domains 
$V_1\Subset V_2\subset \mh_+$ such that 
\[V_1\Supset\big(W^{(1)}\setminus S^{(2)}\big)\cap \mh_+,\;\;\text{and}\;\;V_1\cap J_f\neq\varnothing.\]  
%For every $j\leq s$ let $A_j$ be the connected component of the preimage $F^{-j}(V_1)$ 
% such that $A_j\ni F^{s-j}(z)$. Using induction on $j$ it is easy to show that $A_j\cap J_F\neq \varnothing$
%  for all $j\leq s$.  
Since the postcritical set of $F$ belongs to $\mr$ there exists $\phi:V_2\to\mc$ such that \[F^s\circ\phi=\id\;\;
\text{and}\;\;\phi(F^s(z))=z.\] Notice that $\phi(V_2)\subset W$. By Koebe Distortion Theorem there exists a constant 
$M_1$ (independent from $s$ or $z$) such that
\[|\phi'(u)|\le M_1{|\phi'(F^s(z))|}=\frac{M_1}{|DF^s(z)|}\;\;\text{for all}\;\;u\in V_1.\] The condition  
$F^j(z)\in W$ for $j=0,1,\ldots,s$ implies that $\phi(V_1\cap J_F)\subset J_F$. 
Set \[M_2=\sup\limits_{u\in V_1}\frac{d(u,\overline V_1\cap J_F)}{d(u,J_F)}.\]
Then 
\[d(z,J_f)\le \frac{M_1d(F^s(z),\overline V_1\cap J_f)}{|DF^s(z)|}\le \frac{M_1M_2d(F^s(z),J_f)}{|DF^s(z)|} .\]

The second case can be treated similarly.
\end{proof}
\begin{proof}[Proof of Proposition \ref{PropPolyEsc}.] $1)$ First assume that $j_{term}<\infty$.
Set $r=\inf\{|F'(z)|:z\in W^{(1)}\setminus S^{(2)}\}.$ Clearly, $r>0$. Let $k$ be such that 
$w_k\in W^{(1)}\setminus S^{(2)}$. 
%Then by Proposition \ref{PropTilingStr} $F(w_k)\notin W^{(2)}$. 
%Therefore $m_{k+1}\le m_k$, $r_{k+1}\le r_k$, and thus
Since $r_{k+1}\ge r_k-1$ we have: \[|DH_k(w_k)|=\lambda^{r_k-r_{k+1}}|DF(w_k)|\ge\lambda r.\]

Further, let $I_0$ be the set of indexes $i$ from $0,1,\ldots,j_{term}$ such that 
$w_i\in W^{(1)}\setminus S^{(2)}$. Notice that $j_{term}\in I_0$. Set $I=I_0\cup \{0\}$. 
let $j_1<j_2$ be two consecutive indexes from $I$. Observe that if $w_k\in R^{(2)}$ for some $k$ then $F(w_k)\subset W^{(1)}_+$, 
$m_{k+1}\ge 1$ and thus either $w_{k+1}\in R^{(2)}$ or $k+1\in I$. It follows that there exists $j_1<j\le j_2$ such that $m_k=0$ for $k=j_1+1,j_1+2,\ldots, j-1$ and 
$w_k\in R^{(2)}$ for $k=j,j+1,\ldots, j_2-1$. Using Propositions \ref{PropEuclExpV} and \ref{PropEuclExpP1} we get:
\begin{align*}\big|DH_{j_1,j_2}(w_{j_1})\big|=
\big|DH_{j_1}(w_{j_1})\big|\cdot\big|DH_{j_1+1,j}(w_{j_1+1})\big|\cdot\big|DH_{j,j_2}(w_{j})\big|
\ge\\ \lambda r C_3\lambda_3^{j-j_1-1}C_2\lambda_2^{j_2-j}\ge C_4\lambda_4^{j_2-j_1},\;\;\text{where}\;\;C_4=\lambda rC_2C_3/\lambda_3,\lambda_4=
\min\{\lambda_2,\lambda_3\}.\end{align*}  As before, let $N$ be the number of indexes $k$ such that $m_k\ge 1$ and $w_k\in W^{(1)}\setminus S^{(2)}$. We obtain:
\[|DH_{0,j_{term}}(w_0)|\ge C_4^{N+1}\lambda_4^{j_{term}}.\]
Notice that $m_{j_{term}}=1, z_{j_{term}}=w_{j_{term}}$ and hence \[H_{0,j_{term}}(w)=F^{s_{j_{term}}}(\lambda^{r_0}w).\] It follows that 
\[|DF^{s_{j_{term}}}(z)|\ge C_4^{N+1}\lambda_4^{j_{term}}.\]
Since $r_j=0$ for $j\ge j_{term}$ we have $H_j=F$ for $j\ge j_{term}$. By Proposition \ref{PropEuclExpP1},
\[|DF^{s_{k_{term}}-s_{j_{term}}}(z_{j_{term}})|\ge C_3r\lambda_3^{k_{term}-j_{term}-1}\ge C_4\lambda_4^{k_{term}-j_{term}}.\]
Therefore \[|DF^{s_{k_{term}}}(z)|\ge C_4^{N+2}\lambda_4^{k_{term}}.\]

Assume that $d(z,J_f)\ge 2^{-n}$. Then by Lemma \ref{LmUsingKoebe}, \[|DF^{s_{j_{term}}}(z)|\le 2^nC d(z_{j_{term}},J_f)
 \;\;\text{and}\;\;
|DF^{s_{k_{term}}-s_{j_{term}}}(z_{j_{term}})|\le C\frac{d(z_{k_{term}},J_f)}{d(z_{j_{term}},J_f)}.\] 
Thus, $|DF^{s_{k_{term}}}(z)|\le 2^nM$, where 
$M=C^2\diam(W)$. Therefore, $k_{term}\le A_1n+A_2N+A_3$ where
\[A_1=\log_{\lambda_4}2,A_2=-\log_{\lambda_4}C_4,A_3=\log_{\lambda_4}M-2\log_{\lambda_4}C_4.\]
On the other hand, using Proposition \ref{PropHypExp} we obtain:
\[N\le n\log_{\lambda_1}2+\log_{\lambda_1}(M/C_1).\] Thus, $k_{term}\le An+B$, where 
\[A=A_1+A_2\log_{\lambda_1}2, B=A_3+A_2\log_{\lambda_1}(M/C_1).\]

$2)$ Assume now that $j_{term}=\infty$, that is $z_j\in Q^{(1)}\setminus W^{(1)}$ for all $j<k_{term}$. Then $m_j=0$,
 $z_j=w_j=F^j(z)$ and $H_j=F$ for all $j<k_{term}$.  Using Proposition 
\ref{PropEuclExpP1} we get:
\[|DF^{k_{term}}(z)|\ge C_3\lambda_3^{k_{term}}.\] If $d(z,J_f)\ge 2^{-n}$ then \[|DF^{k_{term}}(z)|\le M2^n\;\;
 \text{and}\;\;k_{term}\le n\log_{\lambda_3}2+\log_{\lambda_3}(M/C_3),\] which finishes the proof.
\end{proof}
 \section{The algorithm.}
 Fix a dyadic number $\delta>0$ such that
\[\delta<\tfrac{1}{2}d(\C\setminus W,W^{(1)})\;\;\text{and}\;\;F(U_\delta(J_f))\subset Q^{(1)}.\]
\begin{Prop}\label{PropUsingKoebe2} There exist constants $K_1,K_2>0$
 such that for any
$z\in U_{\delta}(J_F)$ and any $k\in \mathbb{N}$ if  $F^k(z)\in U_\delta(Q^{(1)})\setminus U_{\delta}(J_F)$ then one has 
$$\frac{K_1}{|DF^k(z)|}\leqslant d(z,J_F)\leqslant \frac{K_2}{|DF^k(z)|}.$$ \end{Prop}
\begin{proof}
There exist a finite
number of pairs of simply connected sets
 $W_j\Subset U_j$ such that the following is true
 \begin{itemize}\item[$1)$] $W_j\cap J_F\neq \varnothing$ for any $j$;
 \item[$2)$] $\bigcup W_j\supset U_{\delta}(Q^{(1)})\setminus U_{\delta}(J_F)$;
 \item[$3)$] $\bigcup U_j\subset W\setminus [-1,1]$.
 \end{itemize}
Assume now that $z,k$ satisfy the conditions of the proposition.
Then $F^k(z)\in W_j$ for some $j$. Since the postcritical set of $F_{|W}$ belongs to $[-1,1]$ the map $F^k$ admits an inverse on $U_j$. Let $\phi:U_j\rightarrow U(z)$ be the
branch of $(F^k)^{-1}$ such that $\phi(F^k(z))=z$. By Koebe Distortion Theorem
 there exists $R_j>0$ independent of $z$ or $k$ such
that $\phi(W_j)$ is contained in the disk of radius
$R_j|D\phi(F^k(z))|$ centered at $z$. It follows that
$$d(z,J_F)\leqslant \frac{R_j}{|DF^k(z)|}.$$ Set $K_2=\max\{R_j\}$. Construction of $K_1$ is similar.
\end{proof}
Fix dyadic sets $U_1,U_2$ such that
\[Q^{(1)}\subset U_1\subset U_{\delta/2}(Q^{(1)}),\;\;U_{\delta/4}(J_f)\subset U_2\subset U_{\delta/2}(J_f).\]

Assume that we would like to verify that a dyadic point $z$ is $2^{-n}$ close to $J_f$. Consider first points
 $z$ which lie outside $U_2$. Fix a dyadic set $U_3$ such that
 $$J_F\subset U_3\Subset U_2.$$ Then we can approximate the distance from a point $z\notin U_2$
 to $J_F$ by the
 distance form $z$ to $U_3$ up to a constant factor.

Now assume $z\in U_2$. The key tool of the algorithm is a sequence of approximations of the iterates 
$z_k=F^{s_k}(z)$. However, the numbers $s_k$ depend
on the levels $m_k$ such that  $z_k\in W ^{(m_k)}\setminus W^{(m_k+1)}$. These levels cannot be computed exactly. Because of this,
 we inductively define approximations $$p_j\approx\tilde z_j=F^{\tilde s_j}(z)$$ of, possibly, different iterates  closely related to $\{z_k\}$.  
\vskip 0.3cm 
\noindent{\bf Construction of $\{p_j\}$.} Set $p_0=z,\tilde s_0=0$. Assume that $p_j$ is constructed. If 
$p_j\in U_1$ let
$ \tilde m_j$ be such that $$U_{\delta/4}\big(\lambda^{-\tilde m_j}p_j\big)\subset
 W,\;\;\text{but}\;\;U_{\delta/2}\big(\lambda^{-\tilde m_j}p_j\big)\nsubseteq W^{(1)}.$$ Notice that 
for some $p_j$ there are two choices of $\tilde m_j$. We fix one of them arbitrarily. Set
$$\tilde r_j=\max\{0,\tilde m_j-1\},\;\;\tilde s_{j+1}=\tilde s_j+2^{\tilde r_j}.$$ Let $p_{j+1}$ be an approximation 
of $F^{\tilde s_{j+1}}(z)$ with precision at least $\lambda^{An+B+3}\delta$, with $A$, $B$ as in Proposition \ref{PropPolyEsc}.

If $p_j\notin U_1$ the sequence $\{p_i\}$ terminates at the index $i=j$.

\vskip 0.3cm
\noindent Observe that $m(F^{\tilde s_j}(z))-1\le\tilde m_j\le m(F^{\tilde s_j}(z)).$
\begin{Lm}\label{LmTildemj} If $d(z,J_f)\ge 2^{-n}$ then there exists a finite sequence $j_0=0<j_1<\ldots<j_{k_{term}}$ such that $\tilde s_{j_i}=s_i$ and $j_{i+1}\le j_i+{An+B}$ for all $i$. 
\end{Lm}
\begin{proof} We will prove existence of $j_i$ by induction. The base is obvious: $\tilde s_0=s_0=0$. Assume that 
$\tilde s_{j_i}=s_i$. Recall that by Corollary \ref{CoPolyEsc} $$m_i\le An+B.$$ By definition of $p_{j_i}$ 
we have $$m_i-1\le \tilde m_{j_i}\le m_i.$$ Thus, if $m_i\le 1$ then $\tilde r_{j_i}=r_i=0$ and $\tilde s_{j_i+1}=s_{i+1}$.
Let $m_i\ge 2$. Set $$k=\max\{l:\tilde s_l\le s_{i+1}\}.$$ Assume that $\tilde s_k<s_{i+1}$. 
We have:
$$s_{i+1}-\tilde s_k<2^{\tilde r_k}.$$
% Therefore, $\tilde r_k>0$ and $\tilde m_k\ge 2$. 
Set $x=F^{\tilde s_k}(z)$.
Since the first return map from $W^{(m_i)}$ to $W^{(m_i-1)}$ 
is $F^{2^{m_i-1}}=F^{s_{i+1}-s_i}$ we have $$ m(x)= m(F^{\tilde s_k-s_i}(z_i))\le m_i-2.$$ Notice that 
$m(x)-1\le\tilde m_k\le m(x)$.
 On the other hand, $F^{s_{i+1}-\tilde s_k}(x)\in W^{(m_i-1)}\subset W^{(m(x)-1)}$, therefore, 
$$s_{i+1}-\tilde s_k\ge 2^{m(x)-1}\ge 2^{\tilde r_k}.$$ This contradiction shows that $\tilde s_k=s_{i+1}$ and finishes 
the proof of existence of $j_i$.

Further, fix $i<k_{term}$. Clearly, if $\tilde r_{j_i}=r_i$
 then $j_{i+1}=j_i+1$. Assume that $\tilde r_{j_i}\neq r_i$. Then $\tilde r_{j_i}=r_i-1$. 
If $$\tilde r_{k+1} =\tilde r_k-1\;\;
\text{for all}\;\; k\ge j_i$$ then $\{p_l\}$ terminates at an index $l\le j_i+m_i\le j_i+An+B$, and so  
$j_{i+1}\le j_i+An+B.$
Otherwise, let
$$k=\min\{l\ge j_i:\tilde r_{k+1}\neq \tilde r_k-1\}.$$ For simplicity, set $r=\tilde r_k$. Observe that 
$$\tilde z_{j_i}=z_i\in W^{(r_i+1)},\;\;\text{therefore}\;\; \tilde z_{k+1}\subset F^{2^{r_i-1}+2^{r_i-2}+\ldots+2^r}(W^{(r_i+1)})\subset \pm\lambda^rF^{2^{r_i-r}-1}(W^{(r_i-r+1)}).$$ Notice that $F^{2^{r_i-r}}(W^{(r_i-r+1)})\subset 
W$, and thus $F^{2^{r_i-r}-1}(W^{(r_i-r+1)})$ lies inside the connected component of $F^{-1}(W)$ containing $0$. This connected component is $Q^{(1)}$.
We obtain:
$$\tilde z_{k+1}\subset Q^{(r+1)},\;\;\text{therefore}\;\; \tilde m_{k+1}\ge r\;\;\text{and}\;\;\tilde r_{k+1}\ge r-1.$$
 From definition of $k$ we conclude that $\tilde r_{k+1}\ge r$. Thus,
$$ 2^{\tilde r_{j_i}}+2^{\tilde r_{j_i+1}}+\ldots+2^{\tilde r_k}+2^{\tilde r_{k+1}}= 
2^{r_i-1}+2^{r_i-2}+\ldots+2^r+2^{\tilde r_{k+1}}\ge 2^{r_i}=s_{i+1}-s_i.$$ It follows that 
$j_{i+1}\le k+1\le j_i+m_i\le j_i+An+B$.
\end{proof}

Consider the following 
\vskip 0.3cm
\noindent{\bf Main subprogram:}\\ \\$i:=1$ \\{\bf while} $i\leqslant
(An+B)^2+2$
{\bf do}\\
$(1)$ Compute dyadic approximations $$p_i\approx
\tilde z_i=F^{\tilde s_i}(z)\;\;\text{introduced above and}\;\; d_i\approx
\left|DF^{\tilde s_i}(z)\right|;$$
\\ $(2)$ Check the inclusion
$p_i\in U_1$:\begin{itemize}\item[$\bullet$] if $p_i\in
U_1$, go to step $(5)$;\item[$\bullet$] if $p_i\notin U_1$,
proceed to step $(3)$;\end{itemize}
$(3)$ Check the inequality $d_i\geqslant K_2 2^n+1$. If true, output $0$ and exit the subprogram, otherwise\\
$(4)$ output $1$ and exit the subprogram.\\
$(5)$ $i\rightarrow i+1$\\
{\bf end while}\\
$(6)$ Output $0$ end exit.\\
{\bf end}
\\
\\
Observe that for all $i$ \begin{align*}\tilde z_{i+1}=F^{\tilde s_{i+1}}(z)=F^{2^{\tilde r_i}}(\tilde z_i)=(-\lambda)^{\tilde r_i}F(\tilde z_i/\lambda^{\tilde r_i}),\;\;\text{and}\\
DF^{\tilde s_{i+1}}(z)=DF^{2^{\tilde r_i}}(\tilde z_i)DF^{\tilde s_i}(z)=DF(\tilde z_i/\lambda^{\tilde r_i})DF^{\tilde s_i}(z).\end{align*} Thus, 
the subprogram runs for at most $(An+B)^2+2$
number of while-cycles each of which consists of $O(n)$ arithmetic operations 
and evaluations of $F$ and $F'$ with
precision $O(n)$ dyadic bits. Hence the running time of the subprogram can be bounded by a polynomial.

\begin{Prop}\label{PropSubprogram} Let $h(n,z)$ be the
output of the subprogram. Then \begin{equation}h(n,z)=\left\{\begin{array}{ll}
1,&\text{if}\;\;d(z,J_f)>2^{-n},\\
0,&\text{if}\;\;d(z,J_f)<K2^{-n},\\\text{either}\;0\;\text{or}\;1,&\text{otherwise},\end{array}
\right.\end{equation}where $K=\frac{K_1}{K_2+1}$, \end{Prop}
\begin{proof} Suppose first that the subprogram
runs the while-cycle $(An+B)^2+2$ times and exits at the step $(6)$. This means that $p_i\in U_1$ and $\tilde z_i\in W$ for
$i=1,\ldots,(An+B)^2+1$. Since $F(W\setminus Q^{(1)})\cap W=\varnothing$
 we get that $\tilde z_i\in Q^{(1)}$ for all $i=1,\ldots,(An+B)^2$. By Proposition \ref{PropPolyEsc} and Lemma \ref{LmTildemj} we obtain that $d(z,J_F)\le 2^{-n}$. Thus if $d(z,J_F)>
2^{-n}$, then the subprogram exits at a step other than $(6)$.

Now assume that for some $i\leqslant (An+B)^2+2$ the subprogram falls into the step $(3)$. Then $$p_{i-1}\in
U_1\;\; \text{and}\;\; p_i\notin U_1.$$ By the choice of $\delta$ we get $\tilde z_{i-1}\in U_\delta(Q^{(1)})\setminus U_\delta(J_F)$. Further, if $d_{i-1}\ge
K_2 2^n+1$, then $|DF^{\tilde s_{i-1}}(z)|\geqslant K_22^n$. By Proposition \ref{PropUsingKoebe2}, $$d(z,J_F)\leqslant
2^{-n}.$$ Otherwise, $|Df^{\tilde s_{i-1}}(z)|\leqslant K_22^n+2\leqslant (K_2+1)2^n.$ In this case Proposition
\ref{LmUsingKoebe} 
 implies that
 $$d(z,J_f)\geqslant \frac{K_1}{K_2+1}2^{-n}.$$
\end{proof}
 Now, to distinguish the case when $d(z,J_f)<2^{-n-1}$ from the case when
$d(z,J_f)>2^{-n}$ we can partition each pixel of size $2^{-n}\times 2^{-n}$ into pixels of size
$(2^{-n}/K)\times (2^{-n}/K)$ and  run the subprogram for the center of each subpixel. This would
increase the running time at most by a constant factor.
 \section{Remarks and open questions.}
We remark that our approach should carry over from the Feigenbaum map $F$ to the Feigenbaum polynomial $f_{c_*}$ (with an oracle for $c_*$)
in a straightforward
fashion, however, at the cost of further complicating what already is a rather technical proof. It should also work for other
infinitely renormalizable real quadratic polynomials with bounded combinatorics. Unbounded combinatorics for real quadratics
is created either by small perturbations of parabolic or of Misiurewicz dynamics (see \cite{L}). Both of these cases is 
computationally ``tame'' (see \cite{Dudko}), hence, we conjecture:

\medskip
\noindent
{\bf Conjecture.} {\sl For every infinitely renormalizable real quadratic polynomial $f_c$ its Julia set is poly-time computable
with an oracle for $c$.}

\medskip
\noindent
Many open questions on computational complexity of quadratic Julia sets remain open. Let us conclude by mentioning two foremost
ones. The first is complexity bounds on Cremer quadratic Julia sets: it is known that all of them are computable \cite{BBY07}, but no
informative pictures have ever been produced. Nothing is known about their computational complexity, in particular, it is not
known if any of them are computationally hard. New ideas and techniques are likely required to make progress here.

The second question has already been formulated in \cite{Dudko}:

\medskip
\noindent
{\bf Question.} {\sl Is the Julia set of a typical real quadratic map poly-time? }

\medskip
\noindent
Weak hyperbolicity is typical in the real quadratic family -- however, it is not
clear to us whether it is sufficient for poly-time computability (see the discussion in \cite{Dudko}). Our renormalization-based
approach developed in the present work may also prove useful in tackling this problem.

\end{large}

\end{document}